\documentclass{amsart}
\usepackage{amsmath,amsfonts,amssymb,amsthm,epsfig, comment}

\newtheorem{Theorem}{Theorem}[section]
\newtheorem{Proposition}[Theorem]{Proposition} 
\newtheorem{Lemma}[Theorem]{Lemma}

\newtheorem{Corollary}[Theorem]{Corollary}
\newtheorem{Conjecture}[Theorem]{Conjecture}

\newcommand{\K}{\mathcal{K}}
\newcommand{\Gr}{\mathcal{G}r}
\newcommand{\D}{\mathrm{D}}
\newcommand{\C}{\mathbb{C}}
\newcommand{\Z}{\mathbb{Z}}
\newcommand{\R}{\mathbb{R}}
\newcommand{\N}{\mathbb{N}}
\newcommand{\val}{\operatorname{val}}
\newcommand{\conv}{\operatorname{conv}}

\newcommand{\tri}{\Delta}
\newcommand{\HIVE}{\mathrm{HIVE}}
\newcommand{\cwl}{\Lambda}
\newcommand{\fund}{\Lambda}
\newcommand{\realt}{\mathfrak{t}_\R}
\newcommand{\MV}{\mathrm{MV}}
\newcommand{\wi}{\mathbf{i}}
\newcommand{\edge}{\longrightarrow}
\newcommand{\Comp}{\operatorname{Comp}}

\theoremstyle{definition}

\begin{document}

\title[Hives and the fibres of the convolution morphism]{Hives and the fibres of the convolution morphism}

\
\author{Joel Kamnitzer}
\email{jkamnitz@aimath.org}
\address{Department of Mathematics\\ UC Berkeley \\ Berkeley, CA }

\begin{abstract}
By the geometric Satake correspondence, the number of components of certain fibres of the affine Grassmannian convolution morphism equals the tensor product multiplicity for representations of the Langlands dual group.  On the other hand, in the case of $GL_n$, combinatorial objects called hives also count tensor product multiplicities.  The purpose of this paper is to give a simple bijection between hives and the components of these fibres.  In particular, we give a description of the individual components.  We also describe a conjectural generalization involving the octahedron recurrence.
\end{abstract}

\date{\today}
\maketitle

\section{Introduction}
\subsection{Tensor product multiplicities and the affine Grassmannian}
Consider the complex reductive group $ G = GL_n $.  Let $ \mathcal{O} = \mathbb{C}[[t]] $ and let $ \mathcal{K} = \mathbb{C}((t)) $.  Let $ \Gr := G(\mathcal{O}) \setminus G(\mathcal{K}) $ denote the \textbf{affine Grassmannian} for $ GL_n$, an ind-scheme over $ \C $.  

The affine Grassmannian is stratified by the $ G(\mathcal{O}) $ orbits $ \Gr_\lambda $ which are labelled by $ \lambda \in \Lambda_+ := \{(\lambda_1, \dots, \lambda_n) : \lambda_1 \ge \dots \ge \lambda_n \} \subset \Z^n $.  Similarly, the $ G(\K) $ orbits on $ \Gr \times \Gr $ are also labelled by $ \Lambda_+ $ and we write $ L_1 \overset{\lambda}{\edge} L_2 $ if $ (L_1,L_2) $ is in the orbit labelled by $ \lambda $.  Let $ L_0 $ denote the identity coset in $ \Gr $.

We can form the \textbf{twisted product} of two $G(\mathcal{O}) $ orbits as 
\begin{equation*}
\Gr_\lambda \widetilde{\times} \Gr_\mu := \{ (L_1, L_2) \in \Gr \times \Gr : L_1 \in \Gr_\lambda \text{ and } L_1 \overset{\mu}{\edge} L_2 \}
\end{equation*}
We have an obvious map $ m_{\lambda \mu} : \Gr_\lambda \widetilde{\times} \Gr_\mu \rightarrow \Gr $ taking $ (L_1, L_2) $ to $ L_2 $.  This map is called the convolution morphism.

The geometric Satake correspondence of Lusztig \cite{L}, Ginzburg \cite{G}, and Mirkovi\'c-Vilonen \cite{MV} is an equivalence between the category of perverse sheaves on $ \Gr $ (constructible with respect to the above stratification) and the category of representations of the Langlands dual group, which in this case is also $ GL_n$.  Under this equivalence, the IC sheaf of $ \overline{\Gr}_\lambda $ corresponds to the irreducible representation $ V_\lambda $ of highest weight $ \lambda $.  Moreover, the push forward under $ m_{\lambda \mu} $ of the IC sheaf of $ \overline{\Gr}_\lambda \widetilde{\times} \overline{\Gr}_\mu $ corresponds to the tensor product $ V_\lambda \otimes V_\mu $.  

As a consequence, the fibres of the convolution morphism record tensor product multiplicities.  
\begin{Theorem}
For all $ \lambda, \mu, \nu \in \Lambda_+$ and any $ L \in \Gr_\nu $ the number of components of $ m_{\lambda \mu}^{-1}(L) $ of dimension $ \langle \rho, \lambda + \mu - \nu \rangle $ equals the tensor product multiplicity of $ V_\nu $ in $ V_\lambda \otimes V_\mu $.
\end{Theorem}

Here $ \rho = (n-1, n-2, \dots, 0) $.  In the case of $ G = GL_n $, Haines \cite[Prop 1.8]{H} has shown that all components of $ m_{\lambda \mu}^{-1}(L) $ are of this dimension.

Both the fibres of the convolution morphism and the tensor product multiplicity problem admit variants which are more symmetric in $ \lambda, \mu, \nu $.

Let 
\begin{equation*}
\Gr_{\lambda\mu\chi} := \{ ( L_1, L_2, L_3) \in \Gr^3 : L_0 \overset{\lambda}{\edge} L_1 \overset{\mu}{\edge} L_2 \overset{\chi}{\edge} L_3, \text{ and } L_0 = L_3 \}.
\end{equation*}
In this definition $ L_0 $ denotes the identity coset of $ \Gr $.

Now let
\begin{equation*}
c_{\lambda\mu\chi} := \dim (V_\lambda \otimes V_\mu \otimes V_\chi)^G 
\end{equation*}

The following is an easy reformulation of the previous theorem.
\begin{Theorem}
The number of components of $ \Gr_{\lambda\mu\chi} $ of dimension $ \langle \rho, \lambda + \mu + \chi \rangle $ equals $ c_{\lambda\mu\chi} $.
\end{Theorem}

Note that $ \Gr_{\lambda \mu \chi} $ is the variety of geodesic triangles in the Bruhat-Tits building for $ G(\K) $ whose vertices are special, whose side lengths are $ \lambda, \mu, \chi $, and whose first vertex is $ L_0 $.  Such triangles have been studied extensively by Kapovich-Leeb-Millson (see for example \cite{KLM}).

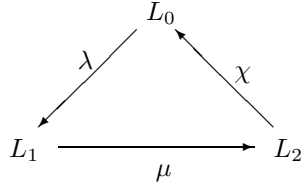
\setlength{\unitlength}{1.3mm} 
\begin{equation*}
\put(-1,0){$L_0$} 
\put(-2,-1){\vector(-1,-1){10}} 
\put(-8,-5){$\lambda$}
\put(12,-11){\vector(-1,1){10}} 
\put(8,-6){$\chi$} 
\put(-15,-14){$L_1$}
\put(-10,-13){\vector(1,0){20}}
\put(12,-14){$L_2$}
\put(0,-16){$\mu$} 
\end{equation*}

\subsection{Hives}
It is a classical problem to give a collection of combinatorial objects of cardinality $c_{\lambda\mu\chi} $.  Many different combinatorial objects can be used; for the purposes of this paper we will consider the hives of Knutson-Tao-Woodward \cite{KTW}, which were inspired by the Berenstein-Zelevinsky triangles \cite{BZ}.

Consider the triangle $\big\{ (i,j,k): i+j+k=n, i,j,k\geq 0 \big\}$.  
This has $\binom{n+2}{2}$ integer points; call this finite set $\tri_n$. 
We will draw it in the plane and put $(n,0,0)$ at the top, 
$(0,n,0)$ at the lower right and $(0,0,n)$ at the lower left.  We will consider the set $ \Z^{\tri_n} $ of integer labelling of these points.
 
We say that $ F \in \Z^{\tri_n}$ satisfies the {\bf hive condition} if:
\begin{equation}
\begin{aligned} \label{rhombin} 
(\text{i})\,\,\qquad F_{i, j, k} + F_{i, j+1 ,k-1} &\ge F_{i+1,j , k-1} + F_{i-1, j+1, k}\,,   \\
(\text{ii})\, \qquad F_{i, j, k} + F_{i+1, j-1, k} &\ge F_{i+1, j, k-1} + F_{i, j-1, k+1}\,,   \\
(\text{iii})  \qquad F_{i, j, k} + F_{i+1, j, k-1} &\ge F_{i, j+1, k-1} + F_{i+1, j-1, k}\,.  \\
\end{aligned}
\end{equation}
 
These inequalities can be interpreted as saying that for any unit rhombus in a hive, the sum across
 the short diagonal is greater than the sum across the long diagonal.  The first two sets of inequalities in (\ref{rhombin}) correspond to horizontally aligned rhombi, while the third set corresponds to vertical rhombi.

A {\bf hive} is an equivalence class of functions satisfying the hive condition, where two functions are considered to be equivalent if their difference is a constant function.  

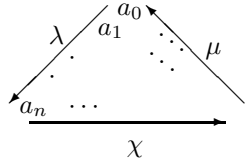
\setlength{\unitlength}{1.3mm} 
\begin{equation*}
\put(-2,0){\vector(-1,-1){10}} 
\put(-8,-4){$\lambda$}
\put(12,-10){\vector(-1,1){10}} 
\put(8,-5){$\mu$} 
\put(-10,-12){\vector(1,0){20}}
\put(0,-15){$\chi$} 
\put(-1, -1){$a_0$}
\put(-3, -3){$a_1$}
\put(-6, -6){$\cdot$}
\put(-8, -8){$\cdot$}
\put(-11, -11){$a_n$}
\put(3, -5){$\ddots$}
\put(2, -7){$\ddots$}
\put(-6, -11){$\cdots$}
\end{equation*}
 
By adding together rhombus inequalities along the left edge of the hive, we see that 
$(\lambda_1 = a_1 - a_0, \ldots, \lambda_n=a_n-a_{n-1}) $  
  is a weakly decreasing sequence of integers and hence is an element of $ \Lambda_+$.  Similarly, the other two edges give elements $ \mu, \chi \in \Lambda_+$.  We refer to these three sequences as the \textbf{boundary} of the hive. 
 
\begin{Theorem}[\cite{KTW}] \label{th:numhives}
The number of hives with boundary $ \lambda, \mu, \chi $ equals $ c_{\lambda\mu\chi} $.
\end{Theorem}

\subsection{Statement of the main result} \label{se:mainthm}
By the above theorems, $c_{\lambda \mu \chi} $ is both the number of components of the variety $ \Gr_{\lambda \mu \chi} $ and the number of hives with boundary values $ \lambda, \mu, \chi $.  Moreover both the points of the variety and the hives have a ``triangular appearance''.  So it is tempting to look for a bijection between this set of components and this set of hives.  Such a bijection is the main result of this paper.

Let $ i, j, k \in \{1, \dots, n\} $ with $ i + j + k = n $.  Consider the tensor product $ W_{ijk} := V_{\fund_i} \otimes V_{\fund_j} \otimes V_{\fund_k} $ of fundamental representations of $ GL_n $ (recall that $ V_{\fund_i} = \Lambda^i \C^n $).  We may view $ W_{ijk} $ either as a representation of $ GL_n $ or as a representation of $ GL_n^3 $.  As a $ GL_n $ representation, it contains a unique one dimensional subrepresentation isomorphic to the determinant representation.  Fix a basis vector $ \xi_{ijk} \in W_{ijk} $ of this subrepresentation.

Define a constructible function $ H :  \Gr_{\lambda\mu\chi} \rightarrow \Z^{\tri_n}$ by
\begin{equation*}
H_{ijk}([g_1], [g_2], [g_3]) := \val \big((g_1, g_2, g_3) \cdot \xi_{ijk} \big).
\end{equation*}
where $ \val $ denotes the usual valuation map $ W_{ijk} \otimes \K \rightarrow \Z $.

In general, suppose that $ X $ is a complex algebraic variety, $ Y \subset  X $ is irreducible subvariety and $ f : X \rightarrow S $ is a constructible function, where $ S $ is any set.  Then there exists a dense constructible subset $ U $ of $ Y $ such that $ f $ is constant on $ U $.  In this situation, the value of $ f $ on $ U $ is called the \textbf{generic value} of $ f $ on $ Y $.  

In particular, $ H $ has a generic value on each component of $ \Gr_{\lambda\mu\chi} $.  The following is our main result.
\begin{Theorem} \label{th:main}
The generic values of $ H $ are all hives with boundary values $ \lambda, \mu, \chi $.  The hives corresponding to each component are different.  Hence we get a bijection from the set of components of $ \Gr_{\lambda\mu\chi} $ to the set of hives with boundary values $ \lambda, \mu, \chi $.
\end{Theorem}
In particular, this theorem gives a way of describing individual components of $ \Gr_{\lambda\mu\chi}$.  The component corresponding to a hive $ F $ is the closure of the locus 
\begin{equation*}
\{ (L_1, L_2, L_3) \in \Gr_{\lambda\mu\chi} : H(L_1, L_2, L_3) = F \}.
 \end{equation*}

Though the statement of Theorem \ref{th:main} does not mention MV cycles, the proof of this theorem involves the theory of MV cycles and polytopes as developed by Anderson \cite{A} and the present author \cite{K1, K2}.

The above function $ H $ is closely related to Speyer's function \cite{S}
\begin{align*}
S_{ijk} : GL_n(\C\{t\})^3 &\rightarrow \mathbb{Q} \\
(g_1, g_2, g_3) &\mapsto \val \Big( [x^i y^j z^k] \det \big( x g_1 + y g_2  + z g_3 \big) \Big).
\end{align*}
Here $ \mathbb{C}\{t\} $ is the field of Puiseux series, and $[x^i y^j z^k] $ denotes the extraction of the coefficient of a monomial.  This function $S_{ijk} $ was the inspiration for our function $H_{ijk}$.  The idea of using $ S $ in order to distinguish the components of $ \Gr_{\lambda \mu \chi} $ was suggested to the author by D. Speyer in 2003.

\subsection*{Acknowledgements}
I would first like to thank David Speyer for the above mentioned suggestion and for other helpful conversations.  I would also like to thank Andre Henriques and Allen Knutson for much discussion on the ideas presented here.  Thanks also to Alexander Goncharov, Mikhail Kapovich, Alexander Postnikov, Arun Ram, and Dylan Thurston for interesting conversations.  Finally, I thank the referee for his very careful reading of this paper.  During the course of this work, I was supported by an NSERC graduate fellowship and an AIM postdoctoral fellowship and I enjoyed the hospitality of the MIT and UC Berkeley mathematics departments.

\section{Background}

We begin by clarifying some our notation from the introduction.

Let $ \K = \C((t)) $ denote the field of Laurent series and let $ \mathcal{O} = \C[[t]] $ denote the ring of power series.  We define  the \textbf{affine Grassmannian} to be the left quotient $ \Gr = G(\mathcal{O}) \setminus G(\K) $.  

Note that $ \cwl := \Z^n $ is the coweight lattice of $ GL_n $.  A coweight $ \mu \in \cwl $ gives a homomorphism $ \mathbb{C}^\times \rightarrow T $ and hence an element of $ \Gr$.  We denote the corresponding element $ t^\mu $.  

If $ \lambda \in \Lambda_+ $, let $ \Gr_\lambda = t^\lambda \cdot G(\mathcal{O}) $.  It has dimension $ 2 \langle \lambda, \rho \rangle $ and these are all the $ G(\mathcal{O}) $ orbits on $ \Gr $.  Given $ L_1, L_2 \in \Gr $, we define
\begin{equation*}
L_1 \overset{\lambda}{\edge} L_2 \Longleftrightarrow (L_1, L_2) \in (L_0, t^\lambda) \cdot G(\K) \Longleftrightarrow [g_2 g_1^{-1}] \in \Gr_\lambda \Longleftrightarrow [g_1 g_2^{-1}] \in \Gr_{\lambda^\vee}
\end{equation*}
where $g_1, g_2 $ are any elements of $ G(\K) $ such that $ [g_1] = L_1, [g_2] = L_2 $ and where $ \lambda^\vee = -w_0 \cdot \lambda = (-\lambda_n, \dots, -\lambda_1)$.

\subsection{Functions on $\Gr $ defined by valuation} \label{se:fun}
To continue our clarification, we will now explain why $ H_{ijk} : \Gr^3 \rightarrow \Z $ is a well-defined function.  It is a special case of a more general construction.  Let $ A $ denote a reductive group over $ \C$ and $ V $ be a finite-dimensional representation of $ A $.  

We now consider the vector space $ V \otimes \K $.  This vector space comes with an increasing filtration $$ \cdots \subset V \otimes t \mathcal{O} \subset V \otimes \mathcal{O} \subset V \otimes t^{-1} \mathcal{O} \subset \cdots $$ and hence we can define a map $ \val : V \otimes \K \rightarrow \Z $ by $ \val(u) = k $ if $ u \in U \otimes t^k \mathcal{O} $ and $ u \notin U \otimes t^{k+1} \mathcal{O} $.

The group $ A(\K) $ acts on $ V \otimes \K $ and the action of the subgroup $ A(\mathcal{O}) $ preserves this filtration and hence preserves the valuation of any element. 

Now pick any vector $ v \in V $.  We can regard $ v = v \otimes 1 $ as a element of $ V \otimes \K $.  Define a function $ \D_{V, v} : A(\K) \rightarrow \Z $, by $ \D_{V,v}(g) = \val \big(g \cdot v \big) $.  This function $ \D_{V,v} $ is invariant under left multiplication by $ A(\mathcal{O}) $ since $ A(\mathcal{O}) $ preserves the valuation of any vector.  Hence $ \D_{V,v} $ descends to a constructible function $ \Gr_A := A(\mathcal{O}) \setminus A(\K) \rightarrow \Z $ which we will also denote by $ \D_{V,v}$.  

Moreover, suppose the vector $ v \in V $ is an eigenvector for a subgroup $ B \subset A $.  Then $ \D_{V, v} $ will be invariant under right multiplication by $ B(\mathcal{O}) $.  To see this, let $ \lambda : B \rightarrow \C^\times $ be the eigenvalue of $ v $.  Then if $ h \in B(\mathcal{O}) $, then $ h \cdot v = \lambda_{\mathcal{O}}(h) v $ where $ \lambda_{\mathcal{O}} : B(\mathcal{O}) \rightarrow \mathcal{O}^\times $ is the map obtained from $ \lambda $ by base change.  Since $ \lambda(h) \in \mathcal{O}^\times $ and so does not change the valuation of any element of $ A(\K) $, we see that 
\begin{equation*}
\D_{V,v}([gh]) = \val \big(g h \cdot v\big) = \val(g \cdot \lambda(h) v) = \val\big( \lambda(h) g \cdot v \big) = \val\big(g \cdot v \big) = \D_{V,v}([g]).
\end{equation*}
(A similar argument shows that if $ \lambda = 1$, then $ \D_{V,v} $ is $ B(\K) $ invariant.)
  
In our situation, $ A = GL_n^3,\ V = W_{ijk},\ v = \xi_{ijk} $.  Note that $ \Gr_A = \Gr^3 $.  Hence $ H_{ijk} $ is well-defined.  Finally, the vector $ \xi_{ijk} $ is an eigenvector for the diagonal $ B = GL_n \subset GL_n^3 $ and hence $ H_{ijk} $ is invariant under right multiplication by the diagonal $ GL_n(\mathcal{O}) $.

\subsection{Fibre and the variety of triangles}
Now let $ \nu = \chi^\vee $.  We would like to compare $ m_{\lambda \mu}^{-1}(t^\nu) $ and $ \Gr_{\lambda \mu \chi} $.  Note that the group $ G(\mathcal{O}) $ acts on $ \Gr_{\lambda \mu \chi } $ and that a fundamental domain for this action is $ \{ (L_1, L_2) \in \Gr_{\lambda \mu \chi} : L_2 = t^{\nu} \} = m_{\lambda \mu}^{-1}(t^{\nu}) $.  Since $ G(\mathcal{O}) $ is connected, there is a bijection between the components of $ \Gr_{\lambda \mu \nu} $ and $ m_{\lambda \mu}^{-1}(t^{\nu})$.   

Our function $ H $ is $ G(\mathcal{O}) $ invariant, so the generic value of $ H $ on a component of $ \Gr_{\lambda \mu \chi } $ will be the same as its generic value on the corresponding component of $ m_{\lambda \mu}^{-1}(t^\nu) $.  Hence to prove Theorem \ref{th:main}, it is enough to prove the analogous result where $ H $ is replaced by its restriction to $ m_{\lambda \mu}^{-1}(t^{\nu}) $.  

So our goal will be to study the components of
\begin{equation*}
\begin{aligned}
m_{\lambda \mu}^{-1}(t^\nu) &= \{ L \in \Gr_\lambda : L t^{-\nu} \in \Gr_{\mu^\vee} \} = \Gr_\lambda \cap  \Gr_{\mu^\vee} t^\nu.
\end{aligned}
\end{equation*}

\subsection{MV cycles and polytopes}

Our main tool for studying these components will be the theory of MV cycles and polytopes.  Let $ W $ denote the Weyl group and let $ N $ denote the unipotent radical of a Borel subgroup of $ G $.

For $ w \in W $, let $ N_w = w N w^{-1} $.  For $ w \in W $ and $ \mu \in \cwl $ define the \textbf{semi-infinite cells}
\begin{equation*}
 S_w^\mu := t^\mu \cdot N_w(\K) \subset \Gr. 
\end{equation*}

Let $ \mu_1, \mu_2 $ be coweights with $ \mu_1 \le \mu_2 $.  A component of $ \overline{S_e^{\mu_1} \cap S_{w_0}^{\mu_2}} $ is called an \textbf{MV cycle} of coweight $(\mu_1, \mu_2) $.  

MV cycles are relevant for us since they are the closures of the components $ m^{-1}_{\lambda\mu} (t^{\nu}) $ of the convolution morphism.  The following result is due to Anderson.

\begin{Theorem}[\cite{A}] \label{th:tpMVcycle}
The MV cycles $ A $ of coweight $ (\nu - \mu, \lambda) $ with $ A \subset \overline{\Gr_\lambda} $  and $ A \subset \overline{\Gr_{\mu^\vee}} t^\nu $ are precisely the closures of the top-dimensional components of $ m_{\lambda\mu}^{-1}(t^\nu) $.
\end{Theorem}

\subsection{BZ data and MV cycles}
We now consider a more explicit description of MV cycles due to the author in \cite{K1}.

Given any collection $ \mu_\bullet = \big( \mu_w \big)_{w \in W} $ of coweights, we can form the \textbf{GGMS stratum}
\begin{equation*} \label{eq:GGMSstratum}
A(\mu_\bullet) := \bigcap_{w \in W} S_w^{\mu_w}.
\end{equation*}

It turns out that every MV cycle is the closure of a GGMS stratum.  To see which closures are MV cycles, we will need a ``dual'' way of looking at these GGMS strata.

Let $ \Gamma = \cup_i W \cdot \fund_i $ be the set of chamber weights.  When $ G = GL_n $, $ W \cdot \fund_i $ can be identified with the set of $i$ element subsets of $ \{1, \dots, n\} $.  So $ \Gamma $ can be identified with the set of proper, non-empty subsets of $ \{1, \dots, n\} $.

Fix a highest weight vector $ v_{\fund_i} $ in each fundamental representation $ V_{\fund_i} $ of $ G $.  For each chamber weight $ \gamma = w \cdot \fund_i $, let $ v_\gamma = \overline{w} \cdot v_{\fund_i} $. Since $ G $ acts on $ V_{\fund_i} $,  $ G(\K) $ acts on $ V_{\fund_i} \otimes \K $.

For each $ \gamma \in \Gamma $ define the function $ \D_\gamma $ by:
\begin{equation} \label{eq:Vdef}
\begin{aligned} 
\D_\gamma: \Gr & \rightarrow \Z \\
[g] &\mapsto \val ( g \cdot v_\gamma)
\end{aligned}
\end{equation}
So $ \D_\gamma = \D_{V_{\fund_i}, v_\gamma} $ in the notation of section \ref{se:fun}.

The functions $ \D_\gamma $ have a simple structure with respect to the semi-infinite cells.  To see this note that if $ \gamma = w \cdot \fund_i $, then $ v_\gamma $ is invariant under $ N_w $.  This immediately implies the following lemma.
\begin{Lemma}[\cite{K1}] \label{th:SmDg}
Let $ w \in W $.
The function $ \D_{w \cdot \fund_i} $ takes the constant value $ \langle \mu, w \cdot \fund_i \rangle $ on $ S_w^\mu $.  In fact,
\begin{equation*}
S_w^\mu = \{ L \in \Gr : \D_{w \cdot \fund_i} (L) = \langle \mu, w \cdot \fund_i \rangle \textrm{ for all } i \}.
\end{equation*}
\end{Lemma}

Let $ M_\bullet $ be a collection of integers.  Then we consider the joint level set of the functions $ \D_\bullet$,
\begin{equation}
A(M_\bullet) := \{ L \in Gr : \D_\gamma(L) = M_\gamma \text{ for all } \gamma \}.
\end{equation}

Lemma \ref{th:SmDg} shows that if $ \mu_\bullet $ is related to $ M_\bullet $ by 
\begin{equation} \label{eq:mufromM}
M_{w \cdot \fund_i} = \langle \mu_w, w \cdot \fund_i \rangle,
\end{equation}
then $ A(\mu_\bullet) = A(M_\bullet)$. 

It is fairly easy to see (\cite{K1}) that this GGMS stratum $ A(M_\bullet) $ will be empty unless the following \textbf{edge inequalities} hold: for each $ w \in W $ and $ i \in I $,
\begin{equation} \label{eq:nondeg}
M_{ws_i \cdot \fund_i} + M_{w \cdot \fund_i} + \sum_{j \ne i} a_{j i} M_{w \cdot \fund_j} \le 0.
\end{equation}

Let $ w \in W, i,j \in I $ be such that $ w s_i > w, ws_j > w $, and $ i \ne j$.  We say that a collection $ \big( M_\gamma \big)_{\gamma \in \Gamma} $ satisfies the \textbf{tropical Pl\"ucker relation} at $ (w, i,j) $ if $a_{ij} = 0 $ or if $ a_{ij} = a_{ji} = -1 $ and
\begin{equation} \label{eq:A2trop}
M_{ws_i \cdot \fund_i} + M_{w s_j \cdot \fund_j} = \min(M_{w \cdot \fund_i} + M_{w s_i s_j \cdot \fund_j} , M_{w s_j s_i \cdot \fund_i} + M_{w \cdot \fund_j} ) .  
\end{equation}

We say that a collection $ M_\bullet = \big( M_\gamma \big)_{\gamma \in \Gamma} $ satisfies the \textbf{tropical Pl\"ucker relations} if it satisfies the tropical Pl\"ucker relation at each $ (w, i, j) $. 

The main result of \cite{K1} is that these tropical Pl\"ucker relations characterize the MV cycles.

A collection $ M_\bullet $ of integers is called a \textbf{BZ datum} of coweight $ (\mu_1, \mu_2) $ if:
\begin{enumerate}
\item $ M_\bullet $ satisfies the tropical Pl\"ucker relations.
\item $ M_\bullet $ satisfies the edge inequalities (\ref{eq:nondeg}).
\item If $ \mu_\bullet $ is the GGMS datum of $ P(M_\bullet) $, then $ \mu_e = \mu_1 $ and $\mu_{w_0} = \mu_2 $.
\end{enumerate}

\begin{Theorem}[\cite{K1}] \label{th:BZcycle}
Let $ M_\bullet $ be a BZ datum of coweight $(\mu_1, \mu_2) $.  Then $ \overline{A(M_\bullet)} $ is an MV cycle of coweight $(\mu_1, \mu_2) $.  Moreover, all MV cycles arise this way.   

In particular, if $ A $ is an MV cycle and $ M_\gamma $ is the generic value of $ \D_\gamma $ on A, then $ M_\bullet $ is a BZ datum.
\end{Theorem}

\subsection{MV polytopes}
There is another combinatorial object to mention at this point.  If $ \overline{A(\mu_\bullet)}$ is an MV cycle of coweight $ (\mu_1, \mu_2) $ (by the above theorem and previous remarks, all are of this form), then the convex hull $ \conv \big(\mu_\bullet \big) $ is called an \textbf{MV polytope} of coweight $(\mu_1, \mu_2) $.  The above considerations show that if $ \mu_\bullet $ and $ M_\bullet $ are related as in (\ref{eq:mufromM}), then the polytope is defined by inequalities involving the $ M_\gamma $,
\begin{equation*}
\conv \big( \mu_\bullet \big) = P(M_\bullet) := \{ \alpha \in \realt : \langle \alpha, \gamma \rangle \ge M_\gamma \text{ for all }  \gamma \}.
\end{equation*}
Moreover, this is a convex polytope with vertices $\mu_\bullet $.

So the MV polytope retains all of the information of the MV cycle and thus we have a bijection from MV cycles to MV polytopes.  The following useful lemma due to Anderson \cite{A} shows which MV cycles lie in the fibre of the convolution morphism.

\begin{Lemma}
Let $ \lambda, \mu, \nu \in \Lambda_+ $.  Let $ A $ be an MV cycle of coweight $ (\nu - \mu, \lambda) $ and $ P $ the associated MV polytope.  

Then $A \subset \overline{\Gr_\lambda}$ if and only if $ P \subset \conv \big( W \cdot \lambda \big)$ and $ A \subset \overline{\Gr_{\mu^\vee}} t^\nu $ if and only if $ P \subset \nu - \conv \big(W \cdot \mu \big)$.
\end{Lemma}

Combining this lemma with Theorems \ref{th:tpMVcycle}, \ref{th:BZcycle}, we see that the closures of the components of $ m^{-1}_{\lambda \mu}(t^\nu) $ are of the form $ \overline{A(M_\bullet)} $  for $ M_\bullet $ from the following set of BZ data:
\begin{equation} \label{eq:defMV}
\begin{aligned}
\MV_{\lambda\mu}^\nu := \{ M_\bullet \text{ a BZ datum } : \
 &M_{w_0 \cdot \fund_i} = \langle \lambda, w_0 \cdot \fund_i \rangle \text{ for all } i, \\
 &M_{\fund_i} = \langle \nu - \mu, \fund_i \rangle \text{ for all } i, \\
&P(M_\bullet) \subset \conv \big( W \cdot \lambda \big) \\
&P(M_\bullet) \subset \nu - \conv \big(W \cdot \mu \big) \}
\end{aligned}
\end{equation}

\begin{Corollary} \label{th:numMV}
$|MV_{\lambda \mu}^\nu| = c_{\lambda\mu\chi}$.
\end{Corollary}

The first and second conditions of (\ref{eq:defMV}) are equivalent to $ P(M_\bullet) $ having coweight $ (\nu - \mu, \lambda) $.

In turns out that the third and fourth conditions of (\ref{eq:defMV}) are difficult to use, even though they can be written out as a sequence of inequalities (see \cite{K1}).  Instead we will use the following consequences.

\begin{Lemma} 
Let $ P(M_\bullet) $ be an MV polytope of coweight $ (\mu, \lambda) $ such that $ P(M_\bullet) \subset \conv \big( W \cdot \lambda \big)$ for some $ \lambda \in \Lambda_+,\ \mu \in \Lambda$. 

Let $ w \in W,\ i \in I $ be such that $ l(ws_i) > l(w) $.  Then $ M_{w \cdot \fund_i} \ge M_{w s_i \cdot \fund_i}$.
\end{Lemma}

\begin{proof}
We may choose a reduced word $ \wi $ for the longest element $ w_0 \in W $ such that for some $ k $, $w_k^\wi = w $ and $ i = \wi_{k+1} $.  By \cite{K2}, the difference $ M_{w_k^\wi \cdot \fund_i} - M_{w_{k + 1}^\wi \cdot \fund_i} $ represents part of the $\wi$-Kashiwara datum for $ P(M_\bullet) $.  In particular this difference is positive.
\end{proof}

This lemma and its proof are a bit surprising.  We have a very straightforward statement about the components of a BZ datum, but its proof relies on interpreting differences of these components as parts of the Kashiwara datum. 

In the case $ G = GL_n $, we can strengthen the lemma.  Let $ \gamma, \delta \in W \cdot \fund_i $ be $ i $ element subsets of $ \{1, \dots, n \}$.  We say that $ \gamma \ge \delta $ if $ \gamma - \delta $ is a sum of positive roots.  This is equivalent to existence of an increasing bijection from $ \gamma \smallsetminus \delta $ to $ \delta \smallsetminus \gamma $.  

\begin{Proposition} \label{th:decr}
Let  $ P(M_\bullet) $ be an MV polytope of coweight $ (\mu, \lambda) $ such that $ P(M_\bullet) \subset \conv \big( W \cdot \lambda \big)$ for some $ \lambda \in \Lambda_+,\ \mu \in \Lambda$. 

Let $ \gamma, \delta \in W \cdot \fund_i $ be such that $ \gamma \ge \delta $.  Then $ M_\gamma \ge M_\delta$.
\end{Proposition}

\begin{proof}
Let $ \{a_1, \dots, a_r \} = \gamma \smallsetminus \delta $ and $ \{b_1, \dots, b_r \} = \delta \smallsetminus \gamma $.  Assume that $a_1 < \dots <  a_r $ and $ b_1 < \dots < b_r $.  By hypothesis we have that $ a_1 < b_1, \dots a_r < b_r $.

First consider the case $ r = 1$, so let $ a = a_1, b = b_1 $.  We may choose $ w \in S_n $ such that $ w(\{1, \dots, i \}) = \gamma $, $ w(i) = a, w(i+1) = b $.  Then since $ a < b $, $ l(ws_i) > l(w) $.  Also by construction $ w \cdot \Lambda_i = \gamma, w s_i \cdot \Lambda_i = \delta $.  So by the lemma we see that $ M_\gamma \ge M_\delta $ as desired.

Now, if $ r > 1$, then we simply apply the above procedure $ r $ times to get a chain of inequalities which shows $ M_\gamma \ge M_\delta$.
\end{proof}

It would be interesting to know if this result carries over to general $ G $.

\section{Proof of the main result}
We now apply this theory to prove our main result.  Everything which follows is specific to $ G = GL_n $.  

It will be convenient for this proof to think of our hives in a ``less symmetric manner''.  We introduce the notation $ \HIVE_{\lambda\mu}^\nu := \HIVE_{\lambda\mu\chi} $, the only difference being that we will read the third edge backwards and hence record the successive differences as $ \nu $.  In particular, we have
\begin{equation*}
\begin{aligned}
&\HIVE_{\lambda\mu}^\nu = \big\{ F \in \Z^{\tri_n} : F \text{ satisfies the rhombus inequalities and } \\
 &F_{n-(k-1), 0, k-1} - F_{n-k, 0, k} = \lambda_k, F_{i-1, n-(i-1), 0} - F_{i, n-i, 0} = \mu_i, \ F_{0, n-(k-1), k-1} - F_{0, n-k, k} = \nu_k \big\} \\
\end{aligned}
\end{equation*}
\subsection{A map from MV polytopes to Hives}
We begin by defining a map $ \Phi : MV_{\lambda\mu}^\nu \rightarrow \HIVE_{\lambda\mu}^\nu $.  We define $ \Phi(M_\bullet) $ to be the hive $ F $ with 
\begin{equation*}
F_{ijk} := M_{ \{k+1, \dots, k+i\} } + \nu_{k+i+1} + \dots + \nu_n.
\end{equation*}

\begin{Proposition} \label{th:isahive}
$\Phi(M_\bullet) $ is actually a hive with boundaries $ \lambda, \mu, \nu $.
\end{Proposition}

\begin{proof}
First, we check the boundary values.
We have $ F_{n-k, 0, k} =  M_{ \{k+1, \dots, n\} }$.  But by the first condition from (\ref{eq:defMV}), we have that $$ M_{ \{k+1, \dots, n\} } = \langle \lambda, w_0 \cdot \Lambda_{n-k} \rangle = \lambda_{k+1} + \dots + \lambda_n. $$  Hence the boundary condition holds along the $\lambda$ edge.

We also have $ F_{i, n-i, 0} = M_{ \{1, \dots, i \} } + \nu_{i+1} + \cdots + \nu_n$.  Using the second condition from (\ref{eq:defMV}), we see that this means that $$ F_{i, n-i, 0} = \nu_1 - \mu_1 + \dots + \nu_i - \mu_i +  \nu_{i+1} + \cdots \nu_n $$ and hence that $ F_{i-1, n-(i-1), 0} - F_{i, n-i, 0} = \mu_i $ as desired.

Finally $ F_{0, n-k, k} = \nu_{k+1} + \dots +\nu_n $ and so the $ \nu $ boundary condition holds as well.

Next we check the rhombus inequalities.
We have 
\begin{equation*}
\begin{aligned}
F_{i, j, k} &+ F_{i, j+1 ,k-1} - (F_{i+1,j , k-1} + F_{i-1, j+1, k}) \\
&= M_{ \{k+1, \dots, k+i\} } + M_{ \{k, \dots, k+i -1 \} } - M_{ \{k, \dots, k+i\} } - M_{ \{k+1, \dots, k+i -1\} }
\end{aligned}
\end{equation*}
and the right hand side is nonpositive by the non-degeneracy inequality and hence the first rhombus inequality (1.i) holds.

For the second rhombus inequality,
\begin{equation*}
\begin{aligned}
F_{i,j,k} &+ F_{i+1, j-1, k} -( F_{i+1, j, k-1} + F_{i, j-1, k+1}) \\
&= M_{ \{k+1, \dots, k+i\} } + M_{ \{k+1, \dots, k+i+1\} } - M_{ \{k, \dots, k+i\} } - M_{\{k+2, \dots, k+i+1\} }.
\end{aligned}
\end{equation*}
Now by the tropical Pl\"ucker relation, we see that
\begin{equation*}
M_{ \{k, \dots, k+i \} } + M_{ \{k+2, \dots, k+i+1\} } \ge M_{ \{ k, k+2, \dots, k+i+1 \} } + M_{ \{k+1 \dots, k+i\} }
\end{equation*}
(in particular (\ref{eq:A2trop}) gives us that the RHS is the min of two terms, one of which is the LHS).  Hence
\begin{equation*}
\begin{gathered}
M_{ \{k+1, \dots, k+i \} } + M_{ \{k+1, \dots, k+i+1 \} } - M_{ \{k, \dots, k+i\} } - M_{ \{k+2, \dots, k+i+1 \} }\\ 
\le M_{ \{k+1, \dots, k+i \} } + M_{ \{k+1, \dots, k+i+1\} } - M_{\{k, k+2, \dots, k+i+1\} } - M_{ \{k+1, \dots, k+i\} } \le 0
\end{gathered}
\end{equation*}
where the last inequality follows from Proposition \ref{th:decr} applied to the pair $ \{k+1, \dots, k+i+1 \} \le \{k, k+2, \dots, k+i+1 \} $.  Hence, we see that the second rhombus inequality (1.ii) holds.

Finally,
\begin{equation*}
\begin{gathered}
F_{i, j, k} + F_{i+1, j, k-1} - ( F_{i, j+1, k-1} + F_{i+1, j-1, k}) \\
= M_{\{ k+1, \dots, k+i\} } + M_{ \{k, \dots, k+i\} } - M_{ \{k, \dots, k+i -1\} } - M_{ \{k, \dots, k+i+1\} } + \nu_{k+i+1} - \nu_{k+i}.
\end{gathered}
\end{equation*}
By the same argument as above (except using that $ P(M_\bullet) \subset \nu - \conv(W \cdot \mu) $), we also see that this expression is non-positive.
\end{proof}

The definition of $ \Phi $ may look a bit ad-hoc, but it is actually a composition of some well-known bijections and inclusions.  First, we take the $ \wi$-Lusztig datum of the MV polytope with respect to the reduced word $ 1 \cdots n-1 1 \cdots n-2 \cdots 1 $ (see \cite{K1}).  Then, we use this Lusztig datum to construct a Gelfand-Tsetlin pattern.  Finally we use this Gelfand-Tsetlin pattern to produce a hive following a well-known construction (see \cite{BZ} or for example \cite{us}).

The map $ \Phi $ is in fact a bijection, since it is clearly injective and we know from Theorem \ref{th:numhives} and Corollary \ref{th:numMV} that $ \HIVE_{\lambda\mu}^\nu $ and $ \MV_{\lambda\mu}^\nu $ each have size $ c_{\lambda\mu\chi}$.  Alternatively it is possible to write down an inverse map, but it is a bit involved to check that the resulting BZ datum satisfies the third and fourth conditions of (\ref{eq:defMV}).

\subsection{The components of the fibres}
Recall the function $ H $ defined in section \ref{se:mainthm}.  First note that $ H $ is a well defined function on $ \Gr_{\lambda\mu\nu} $.

\begin{Proposition}
Let $ M_\bullet \in \MV_{\lambda\mu}^\nu $.  The function $ H$ is constant on $ A(M_\bullet) \times \{t^\nu \} \times \{L_0 \}$ and its value there is $ \Phi(M_\bullet) $.
\end{Proposition}

\begin{proof}
First, note that we can write 
\begin{equation*}
\xi_{ijk} = \sum_{\alpha, \beta, \gamma \ : \ \alpha \cup \beta \cup \gamma = \{1, \dots, n \}} \pm \ v_\alpha \otimes v_\beta \otimes v_\gamma 
\end{equation*}

Let $[g] \in A(M_\bullet) $.  Then
\begin{align*} \label{eq:simple}
H_{ijk}([g], t^\nu, 1) &= \val \Biggl( \sum_{\alpha, \beta, \gamma \ : \ \alpha \cup \beta \cup \gamma = \{1, \dots, n \}} \pm \ g \cdot v_\alpha \otimes t^\nu \cdot v_\beta \otimes v_\gamma \Biggr)\\
&= \min_{\alpha, \beta \ : \ \alpha \cap \beta = \emptyset} M_\alpha + \langle \nu, \beta \rangle
\end{align*}
where we use the fact that if $ v_1, \dots, v_l \in W_{ijk} \otimes \K $ are linearly independent, then $ \val(v_1 + \dots + v_l) = \min( \val(v_1), \dots, \val(v_l)) $.

To analyze the resulting min expression, fix $ \alpha $ for the moment.  Since $ \nu $ is dominant, we have that $ \langle \nu, \beta \rangle \ge \langle \nu, \{ r, \dots, n \} \smallsetminus \alpha \rangle $ (here $ r $ is chosen so that $ \{r, \dots, n \} \smallsetminus \alpha $ has size $ j$).  Hence we may assume that $ \beta \subset \{k+1, \dots, n \} $.

Now, $ \alpha \ge \{ k +1, \dots, n \} \smallsetminus \beta $ and so $ M_\alpha \ge M_{ \{k+1, \dots, n\} \smallsetminus \beta} $ by Proposition \ref{th:decr}.  Hence we may assume that $ \alpha = \{k+1, \dots, n\} \smallsetminus \beta $.

Now we apply a similar trick, except using that $ \nu - P(M_\bullet) \subset \conv( W \cdot \mu) $.  Let $ \delta = \{1, \dots, k, k+i+1, \dots, n \} $.  Then $ -\alpha \ge \delta $.  Now, $ \nu -P(M_\bullet) $ is an MV polytope with BZ datum $ M'_\delta = M_{-\delta} + \langle \nu, \delta \rangle $.  Hence by Proposition \ref{th:decr},
\begin{equation*}
\begin{aligned}
M'_{-\alpha} &\ge M'_\delta\\
\Rightarrow M_{\alpha} + \langle \nu, -\alpha \rangle &\ge M_{-\delta} + \langle \nu, \delta \rangle \\
\Rightarrow  M_\alpha + \langle \nu, \beta \rangle &\ge M_{\{k+1, \dots, k+i\}} + \langle \nu, \{k+i+1, \dots, n\} \rangle 
\end{aligned}
\end{equation*}
and hence we see that to achieve the minimum, we should take $ \alpha = \{k+1, \dots, k+i\} $ and $\beta = \{k+i+1, \dots, n \}$.  In other words, we have $ H_{ijk}(L, t^\nu, L_0) = M_{\{k+1, \dots, k+i\}} + \nu_{k+i+1} + \dots + \nu_n $ for all $ L \in A(M_\bullet)$, as desired.
\end{proof}

\begin{proof}[Proof of Theorem \ref{th:main}]
Let $ X $ be a component of $ \Gr_{\lambda \mu \chi} $.  Then by the results of section 2, $$ \overline{X} \cap \Gr_\lambda \times \{t^\nu\} \times \{L_0 \} = \overline{A(M_\bullet)} \times \{t^\nu \} \times \{L_0 \} $$ for some BZ datum $ M_\bullet \in \MV_{\lambda \mu}^\nu $.  Since $ H $ is $ G(\mathcal{O}) $ invariant, the generic value of $ H $ on $ X $ equals its generic value on this intersection and hence its generic value on $ A(M_\bullet) \times \{t^\nu \} \times \{L_0 \}$.  By the above proposition, this value is $ \Phi(M_\bullet) $ which is a hive by Proposition \ref{th:isahive}.  Hence the generic value of $ H $ on any component is a hive with boundary values $ \lambda, \mu, \chi$.  So we get a map from the set of components to the set of hives.

To see that this gives a bijection, we just note that this map is a composition of the bijections
\begin{equation*}
\Comp(\Gr_{\lambda\mu\chi}) \rightarrow \Comp(m_{\lambda\mu}^{-1}(\nu)) \rightarrow \MV_{\lambda \mu}^\nu \rightarrow \HIVE_{\lambda \mu \chi}.
\end{equation*}
\end{proof}

\section{A conjectural generalization}
\subsection{The variety of $k$-gons}
We will now study the variety of $k$-gons in $ \Gr $.  Let $ \lambda^1, \dots, \lambda^k \in \Lambda_+ $ be $ k $ dominant coweights.  Then we define
\begin{equation*}
\Gr_{\lambda^1 \cdots \lambda^k} = \{ (L_1, \dots, L_k) \in \Gr^k : L_0 \overset{\lambda^1}{\edge} L_1 \overset{\lambda^2}{\edge} \cdots \overset{\lambda^{k-1}}{\edge} L_{k-1} \overset{\lambda^k}{\edge} L_k, \text{ and } L_0 = L_k \}
\end{equation*}
This is the variety of geodesic $k$-gons in the Bruhat-Tits building with side lengths $ \lambda^1, \dots, \lambda^k $, all vertices special, and first vertex $ L_0 $.

As before, from the geometric Satake correspondence we have the following result.
\begin{Theorem}
The number of components of $  \Gr_{\lambda^1 \cdots \lambda^k} $ of dimension $ \langle \rho, \lambda^1 + \dots + \lambda^k \rangle $ equals the dimension of $ \big( V_{\lambda^1} \otimes \cdots \otimes V_{\lambda^k})^{G} $.  
\end{Theorem}

As before, in the case $ G = GL_n $, Haines \cite{H} has recently shown that all components of this variety are of this dimension.

\subsection{The $ k$-hives}
Now we will describe a generalization of hives which is due to Knutson-Tao-Woodward \cite{KTW} in the case when $ k =4$ and to A. Henriques in the general case (personal communication).  We consider the set 
\begin{equation*}
\Delta_n^k  := \{ (i_1, \dots, i_k) \in \mathbb{N}^k : i_1 + \cdots i_k = n \}
\end{equation*}

We say that a function $ F : \Delta_n^k \rightarrow \mathbb{Z} $ satisfies the \textbf{octahedron recurrence} if for any $ v = (v_1, \dots, v_k) \in \N^k$ such that $ v_1 + \cdots + v_k = n-2 $ and for any $  i < j < r < s $ (in the cyclic order on \{1, \dots, k\}), we have
\begin{equation} \label{eq:oct}
\min \big( F_{ v + e_i + e_s} + F_{v + e_j + e_r}, F_{v + e_i + e_j} + F_{v + e_r + e_s} \big) = F_{v + e_i + e_r} + F_{v + e_j + e_s} 
\end{equation}
where $ e_i $ is the vector $ (0, \dots, 1, \dots, 0) $ with a 1 in the $ i$th position and 0s elsewhere.  The name ``octahedron recurrence'' comes from the $ n= 4 $ case where $ \Delta_n^k $ is the set of integer points in a tetrahedron of size $ k$ and we have one condition (\ref{eq:oct}) for each unit octahedron in $ \Delta_n^k $.

As before, two functions are considered equivalent if their difference is constant.  A $k$-\textbf{hive} is an equivalence class of functions $ F $ which restrict to hives on all of their 2-faces and satisfy the octahedron recurrence.

The boundary value of a hive is $ \lambda^1, \dots, \lambda^k $ where $ \lambda^i_j = F_{0,\dots, j-1, n-j+1, \dots 0} - F_{0, \dots, j, n-j, \dots, 0} $, so $ \lambda^i $ records the successive differences along the edge from $ ne_{i-1} $ to $ ne_i $.  

The case of 4-hives has studied by Knutson-Tao-Woodward.
\begin{Theorem}[\cite{KTW}]
The number of 4-hives with boundary values $\lambda^1, \lambda^2, \lambda^3, \lambda^4$ equals the dimension of $ (V_{\lambda^1} \otimes V_{\lambda^2} \otimes V_{\lambda^3} \otimes V_{\lambda^4})^G$
\end{Theorem}

The general result that $k$-hives count tensor product invariants seems to be known to experts, though no proof appears in the literature.  A combinatorial model related to $k$-hives has been developed by A. Postnikov (personal communication).

A ``non-tropical'' version of the octahedron recurrence (\ref{eq:oct}) appears in the work of Fock-Goncharov \cite{FG}, where it describes relations between coordinates on the product of $k $ copies of the base affine variety for $ GL_n$.  

\subsection{Components of the variety of $k$-gons and $k$-hives}
For $ (i_1, \dots, i_k) \in \Delta_n^k $, we define $ \xi_{i_1 \cdots i_k} $ to be a basis vector for the copy of the determinant representation inside $ V_{\fund_{i_1}} \otimes \cdots \otimes V_{\fund_{i_k}} $.  

Define the function
\begin{equation*}
H_{i_1 \cdots i_k}([g_1], \dots, [g_k]) := \val \big( (g_1, \dots ,g_k) \cdot \xi_{i_1 \cdots i_k} \big).
\end{equation*}
The general setup of section \ref{se:fun} applies to show that $ H_{i_1 \cdots i_k} $ is a well-defined function $ Gr^k \rightarrow \Z $.  We will consider the restriction of $ H $ to the subvariety $ \Gr_{\lambda^1 \cdots \lambda^k} $.

\begin{Conjecture} \label{conj}
The generic value of $ H $ on each component of $ \Gr_{\lambda^1 \cdots \lambda^k} $ is a $k$-hive and this gives a bijection between the set of components of $ \Gr_{\lambda^1 \cdots \lambda^k} $ and the set of $k$-hives with boundary values $ \lambda^1, \dots, \lambda^k$.
\end{Conjecture}

As supporting evidence for this conjecture, let us mention that the equation (\ref{eq:oct}) can be see as the tropicalization of a equation involving minors of matrices in $ G(\K) $ (the same minors as in \cite{FG}).  A similar observation lead to the tropical Plucker relations in \cite{K1}.

\subsection{An application of the conjecture}
There is an interesting application of the $ k=4$ case of the conjecture (this is the first open case).  As shown in \cite{KTW}, looking at the faces of all 4-hives with boundary $ \lambda, \mu, \nu, \chi $, gives a bijection
\begin{equation*}
\bigcup_\delta \HIVE_{\lambda\delta}^\chi \times \HIVE_{\mu\nu}^\delta \stackrel{\sim}{\longrightarrow}
\bigcup_\gamma \HIVE_{\lambda\mu}^\gamma \times \HIVE_{\gamma\nu}^\chi
\end{equation*}
In \cite{us}, we showed that this bijection realizes the associativity constraint in the category of $ \mathfrak{gl}_n$-crystals.  Combining Conjecture \ref{conj} with an extension of the work of Braverman-Gaitsgory \cite{BG}, would allow one to reprove this result in a geometric manner.  This would be a improvement over the current proof which proceeds combinatorially via the theory of Young Tableaux.


\begin{thebibliography}{E-G-S}
\bibitem[A]{A}
J. E. Anderson, A polytope calculus for semisimple groups, \textit{Duke Math. J.} \textbf{116} (2003), no. 3, 567–-588; math.AG/0110225.

\bibitem[BZ]{BZ} A. Berenstein and A. Zelevinsky, Involutions on Gel 'fand-Tsetlin schemes and multiplicities in skew GLn-modules. (Russian) \textit{Dokl. Akad. Nauk SSSR} \textbf{300} (1988), no. 6, 1291–1294; translation in \textit{Soviet
Math. Dokl.} \textbf{37} (1988), no. 3, 799–802.

\bibitem[BG]{BG}
A. Braverman and D. Gaitsgory, Crystals via the Affine Grassmannian, \textit{Duke Math. J.}, \textbf{107} (2001) no. 3, 561--575; math.AG/9909077.

\bibitem[H]{H}
T. J. Haines, Structure constants for Hecke and representation rings, \textit{IMRN}, no. 39, 2103--2119 (2003); math.RT/0304176.

\bibitem[HK]{us}
A. Henriques and J. Kamnitzer, The octahedron recurrence and gln crystals, \textit{Adv. Math.} \textbf{206} no. 1, 211--249; math.CO/0408114.

\bibitem[FG]{FG}
V. V. Fock and A. B. Goncharov,  Moduli spaces of local systems and higher Teichm\"uller theory, \textit{Publ. Math. IHES} no. 103 (2006) 1-211; math.AG/0311149.

\bibitem[G]{G}
V. Ginzburg, Perverse sheaves on a loop group and Langlands duality; math.AG/\-9511007.

\bibitem[K1]{K1}
J. Kamnitzer, Mirkovi\'c-Vilonen cycles and polytopes; math.AG/0501365.

\bibitem[K2]{K2}
J. Kamnitzer, The crystal structure on the set of Mirkovi\'c-Vilonen polytopes, \textit{Adv. Math.} \textbf{215} no. 1, (2007) 66--93; math.QA/0505398.

\bibitem[KLM]{KLM}
M. Kapovich, B. Leeb, and J. Millson, The generalized triangle inequalities in symmetric spaces and buildings with applications to algebra, to appear in \textit{Mem. AMS}; math.RT/0210256.

\bibitem[KTW]{KTW} 
A. Knutson, T. Tao and C. Woodward, A positive proof of the Littlewood-Richardson rule using the octahedron  
recurrence, \textit{Electron. J. Combin.} \textbf{11} (2004), arXiv:math.CO/0306274.

\bibitem[L]{L}
G. Lusztig, Singularities, character formulas and a q-analog of weight multiplicities, \textit{Ast\'erisque} \textbf{101-102} (1983), 208-229.

\bibitem[MV]{MV}
I. Mirkovi\'c and K. Vilonen, Geometric Langlands duality and representations of algebraic groups over commutative rings, to appear in \textit{Ann. Math.}; math.RT/0401222.

\bibitem[S]{S}
D. Speyer, Horn's Problem, Vinnikov Curves, and the Hive Cone, \textit{Duke Math. J.} \textbf{127} (2005) no. 3, 395--427; math.AG/\-0311428.

\end{thebibliography}
\end{document}